\documentclass[reqno,a4paper]{pjm}

\usepackage[active]{srcltx}
\usepackage{amsmath,amssymb,latexsym}
\usepackage{calc}
\usepackage{calrsfs}
\usepackage{comment,color}
\usepackage{enumerate}
\usepackage{enumitem}
\usepackage{hyperref}
\usepackage{ifthen}
\usepackage{xspace}

\numberwithin{equation}{section}

\usepackage{url}
\newtheorem{lem}{Lemma}
\newtheorem{thm}{Theorem}

\renewcommand\Re{\operatorname{Re}}
\renewcommand\Im{\operatorname{Im}}

\newcommand{\Alg}{\mathbf{A}}
\newcommand{\A}{\mathbf{A}}
\newcommand{\Ban}{X}
\newcommand{\CC}{\mathbb{C}}
\newcommand{\Cos}{C}
\newcommand{\N}{\mathbb{N}}
\newcommand{\R}{\mathbb{R}}
\newcommand{\T}{\mathbb{T}}
\newcommand{\Z}{\mathbb{Z}}
\newcommand{\co}[2][]{c_{#1}^{(#2)}}

\begin{document}

\title[Cosine families]{On cosine families close to scalar cosine families}

\author[W. Chojnacki]{Wojciech Chojnacki}

\address{School of Computer Science\\
The University of Adelaide \\
SA 5005 \\
Australia}
\email{wojciech.chojnacki@adelaide.edu.au}

\address{Wydzia{\l} Matematyczno-Przyrodniczy, 
  Szko{\l}a Nauk \'Scis{\l}ych \\
  Uniwersytet Kardyna\l{}a Stefana Wyszy\'nskiego \\
  Dewajtis 5,  01-815 Warszawa \\ Poland}
\email{w.chojnacki@uksw.edu.pl}

\keywords{Cosine  family,  cosine  sequence,  normed  algebra,  scalar
  element}

\renewcommand{\subjclassname}{\textnormal{2010} Mathematics Subject Classification}

\subjclass{Primary 47D09, 47C05; Secondary 47A30}

\begin{abstract}
  We prove  that if two normed-algebra-valued  cosine families indexed
  by a  single Abelian group,  of which  one is bounded  and comprised
  solely of scalar elements of  the underlying algebra, differ in norm
  by less  than $1$ uniformly  in the parametrising index,  then these
  families coincide.
\end{abstract}

\maketitle

\section{Introduction}
\label{sec:introduction}

A classic result,  in its early form due  to Cox \cite{cox66:_matric},
states that if $\Alg$ is a normed  algebra with a unity denoted by $1$
and $a$ is an element of $\Alg$ such  that $\sup_{n \in \N} \| a^n - 1
\| <  1$, then $a  = 1$.  Cox's version  concerned the case  of square
matrices  of  a  given  size.   This was  later  extended  to  bounded
operators    on    Hilbert    space   by    Nakamura    and    Yoshida
\cite{nakamura67:_cox},  and   to  an  arbitrary  normed   algebra  by
Hirschfeld \cite{hirschfeld68:_banac} and Wallen \cite{wallen67}.  The
latter author proved in fact a stronger result, namely that
\begin{math}
  \| a^n - 1 \| = o(n)
\end{math}
and
\begin{math}
  \liminf_{n \to \infty} n^{-1} \left( \| a - 1 \| + \|a^2 - 1 \| +
    \dots + \| a^n - 1 \| \right) < 1
\end{math}
imply $a = 1$, and he achieved this by using an elementary argument.

An           immediate           consequence          of           the
Cox--Nakamura--Yoshida--Hirschfeld--Wallen   theorem    is   that   if
$\{S(t)\}_{t \geq  0}$ is a  semigroup on  a Banach space  $\Ban$ such
that
\begin{displaymath}
  \sup_{t \geq 0} \| S(t) - I_{\Ban} \| < 1,
\end{displaymath}
then $S(t)  = I_{\Ban}$ for each  $t \geq 0$; here  $I_{\Ban}$ denotes
the identity  operator on  $\Ban$.  Recently, Bobrowski  and Chojnacki
\cite{bobrowski13:_isolat_point_set_bound_cosin}     established    an
analogue of this  result for one-parameter cosine families:  if $a \in
\R$  and $\{\Cos(t)\}_{t  \in  \R}$ is  a  strongly continuous  cosine
family on a Banach space $\Ban$ such that
\begin{equation}
  \label{eq:1}
  \sup_{t \in \R} \| \Cos(t) - (\cos at )I_{\Ban} \| < \frac{1}{2},
\end{equation}
then  $\Cos(t) =  (\cos at  )  I_{\Ban}$ for  each $t  \in \R$.   This
conclusion   was   further   refined   by   Schwenninger   and   Zwart
\cite{schwenninger:_less} who  showed that condition  \eqref{eq:1} can
be replaced by the condition
\begin{displaymath}
  \sup_{t \in \R} \| \Cos(t) - (\cos at )I_{\Ban} \| < 1.
\end{displaymath}
For the  case $a  = 0$, the  same authors later  showed that  the even
weaker condition
\begin{displaymath}
  \sup_{t \in \R} \| \Cos(t) - I_{\Ban} \| < 2
\end{displaymath}
already suffices  \cite{schwenninger:_zero}.  The result  of Bobrowski
and  Chojnacki and  those of  Schwenninger  and Zwart  rely on  rather
involved  arguments,  drawing  on   ideas  from  operator  theory  and
semigroup theory.

In this  note we  extend the  first result  of Schwenninger  and Zwart
(that is, the  result of \cite{schwenninger:_less}) to  cover the case
of cosine families  that are not necessarily indexed  by real numbers,
not necessarily  operator-valued, and not necessary  continuous in any
particular sense.  A crucial step  towards proving the relevant result
will    be    the    establishment    of   an    analogue    of    the
Cox--Nakamura--Yoshida--Hirschfeld--Wallen    theorem    for    cosine
sequences.  The  proof of the  latter result will use  only elementary
means.

\section{Preliminaries and results}
\label{sec:prel-results}
 
Let $\Alg$ be a normed algebra, real or complex, with a unity $1$.  An
element of $\Alg$  is called \emph{scalar} if it is  a scalar multiple
of the unity $1$.  A family in $\Alg$ is termed \emph{scalar} if every
member of this family is scalar.  Given a scalar $\lambda$, the symbol
$\lambda$ will  be employed to denote  both the scalar itself  and the
element of $\Alg$ obtained by  multiplying the unity element of $\Alg$
by  $\lambda$.  In  particular,  if $\{\lambda_{\gamma}\}_{\gamma  \in
  \Gamma}$ is a family  of scalars, then $\{\lambda_{\gamma}\}_{\gamma
  \in \Gamma}$  will also  denote the  corresponding family  of scalar
elements of $\Alg$.

We recall  that an $\Alg$-valued family  $\{a_{\lambda}\}_{\lambda \in
  \Lambda}$  is  said  to  be  \emph{bounded}  if  $\sup_{\lambda  \in
  \Lambda} \| a_{\lambda} \| < \infty$.

Let  $G$ be  an  Abelian  group, written  additively,  with a  neutral
element $0$.  A  family $\{\Cos(g)\}_{g \in G}$ in $\Alg$  is called a
\emph{cosine family} if
\begin{itemize}
\item [(i)] $2  \Cos (g) \Cos (h)  = \Cos (g+h) + \Cos  (g-h)$ for all
  $g,h \in G$ (d'Alembert's functional equation, also called the cosine
  functional equation), 
\item [(ii)] $\Cos(0) = 1$.
\end{itemize}

With  this minimum  of preparation,  we are  ready to  state the  main
result of the paper.

\begin{thm}
  \label{thm:1}
  Let $\A$ be a normed algebra with a unity $1$, let $G$ be an Abelian
  group, let  $\{c(g)\}_{g \in G}$  be a bounded  scalar-valued cosine
  family, and  let $\{\Cos(g)\}_{g  \in G}$  be an  $\A$-valued cosine
  family such that
  \begin{displaymath}
    \sup_{g \in G} \| \Cos(g) - c(g) \| < 1.
  \end{displaymath}
  Then
  \begin{math}
    \Cos(g) = c(g)
  \end{math}
  for each $g \in G$.
\end{thm}

We shall deduce this theorem  from a seemingly weaker result presented
below as Theorem~\ref{thm:2}. 

We  continue  with  preliminary  definitions and  results.   A  cosine
sequence  is a  cosine  family for  which the  indexing  group is  the
additive group  of integers  $\Z$.  
Every cosine sequence $\{c_n\}_{n\in\Z}$ is even: the equality $c_{-n}
= c_n$ holds  for all $n \in \Z$.  Furthermore,  every cosine sequence
$\{c_n\}_{n\in\Z}$ is  uniquely determined  by its element  indexed by
$1$, namely,
\begin{displaymath}
  \label{eq:2}
  c_n = T_{|n|}(c_1)
  \quad
  (n \in \Z),
\end{displaymath}
where,  for $n \in  \N \cup  \{0\}$, $T_n(x)$  is the  $n$th Chebyshev
polynomial of the first kind
\begin{displaymath}
  T_n(x)=\sum_{k=0}^{[n/2]}\binom{n}{2k}x^{n-2k}(x^2-1)^k.
\end{displaymath}
The element $c_1$ of a  cosine sequence $\{c_n\}_{n\in\Z}$ is commonly
termed the  \emph{generator} of the  sequence.  Every element  of $\A$
generates a unique cosine sequence.   The cosine sequence generated by
$a \in \A$ is given by $c_n(a) = T_{|n|}(a)$ for every $n \in \Z$.

For each $\gamma \in  \CC \setminus \{0\}$, let $\{\co[n]{\gamma}\}_{n
  \in \Z}$ be the $\CC$-valued cosine sequence given by
\begin{displaymath}
  \co[n]{\gamma} = \frac{\gamma^n + \gamma^{-n}}{2}
  \quad
  (n \in \Z).
\end{displaymath}
Denote by $\T$  the unit circle in the complex  plane.  If $\gamma \in
\T$, then
\begin{math}
  \co[n]{\gamma} = \Re \gamma^n
\end{math}
for all $n \in \Z$, so that all elements of $\{\co[n]{\gamma}\}_{n \in
  \Z}$ are real numbers with modulus no greater than $1$.

We are  now ready to  present the result  that is the  main ingredient
needed to prove Theorem~\ref{thm:1}.

\begin{thm}
  \label{thm:2}
  Let $\A$ be a normed algebra with  a unity $1$, let $\gamma \in \T$,
  and let $\{c_n\}_{n \in \Z}$  be an $\A$-valued cosine sequence such
  that
  \begin{displaymath}
    \sup_{n \in \Z} \| c_n - \co[n]{\gamma} \| < 1.
  \end{displaymath}
  Then
  \begin{math}
    c_n = \co[n]{\gamma}
  \end{math}
  for each $n \in \Z$.
\end{thm}

Theorem~\ref{thm:2}   can  be   viewed   as  a   counterpart  of   the
Cox--Nakamura--Yoshida--Hirschfeld--Wallen    theorem    for    cosine
sequences.   Its proof  will be  much  in the  spirit of  the work  of
Wallen, although the details will be more complicated.

\section{Proof of Theorem~\ref{thm:2}}
\label{sec:main-result}

This section is  devoted to proving Theorem~\ref{thm:2}.   We begin by
establishing a key algebraic identity.

Let $\A$ be a  normed algebra with a unity $1$, let  $\gamma \in \T$, and
let $\{c_n\}_{n  \in \Z}$ be  an $\A$-valued cosine sequence.   Then
\begin{align}
  \label{eq:3}
   2 (\co[1]{\gamma} - c_1)\sum_{k=0}^{n-1} \gamma^k c_k
   =
   \gamma^n c_{n-1} - \gamma^{n-1} c_n - c_1 + \gamma^{-1}.
\end{align}
Indeed, by the cosine functional equation,
\begin{displaymath}
  2 c_1 \sum_{k=0}^{n-1} \gamma^k c_k = 
  \sum_{k=0}^{n-1} \gamma^k c_{k-1} + \sum_{k=0}^{n-1}
  \gamma^k c_{k+1} 
  = \sum_{k=-1}^{n-2} \gamma^{k+1} c_k + \sum_{k=1}^n
  \gamma^{k-1} c_k.
\end{displaymath}
On the other hand,
\begin{displaymath}
  2 \co[1]{\gamma} \sum_{k=0}^{n-1} \gamma^k c_k
  = (\gamma + \gamma^{-1}) \sum_{k=0}^{n-1} \gamma^k c_k 
  = \sum_{k=0}^{n-1} \gamma^{k+1} c_k + \sum_{k=0}^{n-1} \gamma^{k-1} c_k.
\end{displaymath}
Hence
\begin{align*}
  2 (\co[1]{\gamma} - c_1)\sum_{k=0}^{n-1} \gamma^k c_k
  & =
  \left(\sum_{k=0}^{n-1} - \sum_{k=-1}^{n-2}\right) \gamma^{k+1} c_k
  +
  \left(\sum_{k=0}^{n-1} - \sum_{k=1}^n\right) \gamma^{k-1} c_k
  \\
  & = \gamma^n c_{n-1} - c_{-1} - \gamma^{n-1} c_n + \gamma^{-1},
\end{align*}
which, given that $c_{-1} = c_1$, immediately yields \eqref{eq:3}.

We  next observe  that  the  sequence $\{c_n\}_{n  \in  \Z}$ from  the
statement of  Theorem~\ref{thm:2} is necessarily bounded,  because the
sequence $\{\co[n]{\gamma}\}_{n  \in \Z}$  from the same  statement is
bounded and 
\begin{math}
    \sup_{n \in \Z} \| c_n - \co[n]{\gamma} \| < 1.
\end{math}

\begin{lem}
  \label{lem:1}
  Under the  assumptions of Theorem~\ref{thm:2},  if, for each  $n \in
  \N$, we put
  \begin{displaymath}
    P_n := \frac{1}{n} \sum_{k=0}^{n-1} c_k,
  \end{displaymath}
  then $P_n = 1$ for each $n \in \N$ if $\gamma = 1$, and $\lim_{n \to
    \infty} P_n = 0$ if $\gamma \neq 1$.
\end{lem}

\begin{proof}
  Let $0 < \delta < 1$ be such that
  \begin{math}
    \| c_n - \co[n]{\gamma} \| \leq \delta
  \end{math}
  for each  $n \in  \Z$. We break the proof up into three  cases. 
  
  \textsc{Case} $\gamma = 1$.  Assuming  $\gamma = 1$ in \eqref{eq:3},
  we obtain, for each $n \in \N$,
  \begin{displaymath}
    2(1 - c_1) \sum_{k=0}^{n-1} c_k = c_{n-1} - c_n - c_1 + 1.
  \end{displaymath}
  We can rewrite this as
  \begin{equation}
    \label{eq:4}
    (1 - c_1) P_n = e_n
    \quad
    \text{with
    $e_n = \frac{1}{2n}(c_{n-1} - c_n - c_1 + 1)$}.
  \end{equation}
  Note that
  \begin{math}
    \lim_{n \to \infty} e_n = 0,
  \end{math}
  as $\{c_n\}_{n \in \Z}$ is bounded.   Given that $\co[k]{1} = 1$ for
  every $k \in \Z$, we have, for each $n \in \N$,
  \begin{displaymath}
    1 - P_n = \frac{1}{n} \sum_{k=0}^{n-1} (\co[k]{1} - c_k)
  \end{displaymath}
  and hence
  \begin{displaymath}
    \| 1 - P_n \| \leq \frac{1}{n} \sum_{k=0}^{n-1} \| \co[k]{1} -
    c_k \|
    \leq \delta.
  \end{displaymath}
  Writing
  \begin{displaymath}
    1 - c_1 = (1 - c_1) (1 - P_n) + e_n,
  \end{displaymath}
  we find that
  \begin{displaymath}
    \| 1 - c_1 \|  \leq \delta \|1 - c_1 \| + \| e_n \|,
  \end{displaymath}
  whence, letting $n \to \infty$,
  \begin{displaymath}
    \| 1 - c_1 \|  \leq \delta \|1 - c_1 \|.
  \end{displaymath}
  As $\delta <  1$, we see that $\|  1 - c_1 \| = 0$  and further that
  $c_1 = 1$.   Consequently, $c_n = 1$  for each $n \in  \Z$, and thus
  $P_n = 1$ for each $n \in \N$.

  \textsc{Case} $\gamma = -1$. Assuming $\gamma = -1$ in \eqref{eq:3},
  we deduce that, for each $n \in \N$,
  \begin{displaymath}
    2 (1 + c_1)\sum_{k=0}^{n-1} (-1)^k c_k
    =
    (-1)^{n-1}(c_n + c_{n-1}) + c_1 + 1.
  \end{displaymath}
  Letting
  \begin{displaymath}
    Q_n := \frac{1}{n} \sum_{k=0}^{n-1} (-1)^k c_k,
  \end{displaymath}
  we see that
  \begin{displaymath}
    (1 + c_1) Q_n = f_n
  \quad
  \text{with
  $f_n = \frac{1}{2n}((-1)^{n-1}(c_n + c_{n-1}) + c_1 + 1)$}.
  \end{displaymath}
  Clearly,  $\lim_{n \to  \infty}f_n =  0$. Taking  into account  that
  $\co[k]{-1} = (-1)^k$ for every $k \in \Z$, we have, for each $n \in
  \N$,
  \begin{displaymath}
    1 - Q_n = \frac{1}{n} \sum_{k=0}^{n-1} (-1)^k (\co[k]{-1} - c_k)
  \end{displaymath}
  and hence
  \begin{displaymath}
    \| 1 - Q_n \| \leq
     \frac{1}{n} \sum_{k=0}^{n-1} \| \co[k]{-1} - c_k \|
    \leq \delta.
  \end{displaymath}
  Writing
  \begin{displaymath}
    1 + c_1 = (1 + c_1) (1 - Q_n) + f_n,
  \end{displaymath}
  we see that
  \begin{displaymath}
    \| 1 + c_1 \|  \leq \delta \| 1 + c_1 \| + \| f_n \|,
  \end{displaymath}
  whence, letting $n \to \infty$,
  \begin{displaymath}
    \| 1 + c_1 \|  \leq \delta \| 1 + c_1 \|.
  \end{displaymath}
  As $\delta <  1$, we conclude that $\|  1 + c_1 \| =  0$ and further
  that $c_1 = - 1$.  Consequently, $c_n = (-1)^n$ for each $n \in \Z$,
  and thus
  \begin{displaymath}
    P_n = \frac{1 + (-1)^{n-1}}{2n}
  \end{displaymath}
  for each $n \in \N$, which immediately yields
  \begin{math}
    \lim_{n \to \infty} P_n = 0.
  \end{math}

  \textsc{Case} $\gamma  \notin \{-1,1\}$.   In this case  $\Im \gamma
  \neq 0$, and,  as $|\co[1]{\gamma}| = |\Re \gamma| =  \sqrt{1 - (\Im
    \gamma)^2}$, we have $| \co[1]{\gamma} | < 1$.  Let $\epsilon > 0$
  be  such  that  $\delta  +  \epsilon   <  1$.   Choose  $l  \in  \N$
  sufficiently large so that $| \co[1]{\gamma} |^l < \epsilon$.  It is
  readily proved by induction that
  \begin{displaymath}
    c_1^l = 
    \frac{1}{2^{l-1}} \sum_{k=0}^{\frac{l-1}{2}} \binom{l}{k} c_{l-2k}
  \end{displaymath}
  if $l$ is odd, and
  \begin{displaymath}
    c_1^l =
    \frac{1}{2^l} \binom{l}{\frac{l}{2}} + \frac{1}{2^{l-1}}
    \sum_{k=0}^{\frac{l}{2}-1} \binom{l}{k} c_{l-2k}
  \end{displaymath}
  if   $l$    is   even,    with   similar   formulae    holding   for
  $(\co[1]{\gamma})^l$  (cf.   \cite[formulae   1.320  5  and
  1.320 7]{gradshteyn07:_table}). Hence
  \begin{displaymath}
    \Big \| c_1^l - (\co[1]{\gamma})^l \Big \|
    \leq
    \frac{1}{2^{l-1}} \sum_{k=0}^{\frac{l-1}{2}} \binom{l}{k}
    \left\| c_{l - 2k} - \co[l - 2k]{\gamma} \right\|
    \leq
    \delta
  \end{displaymath}
  if $l$ is odd, and
  \begin{displaymath}
    \left\| c_1^l - (\co[1]{\gamma})^l \right\|
    \leq
    \frac{1}{2^{l-1}} \sum_{k=0}^{\frac{l}{2} -1} \binom{l}{k}
    \left\| c_{l - 2k} - \co[l - 2k]{\gamma} \right\|
    \leq
    \left(1 - \frac{1}{2^l}\binom{l}{\frac{l}{2}} \right)\delta
    <
    \delta
  \end{displaymath}
  if $l$ is even. In either case
  \begin{math}
    \| c_1^l - (\co[1]{\gamma})^l \| \leq \delta.
  \end{math}
  It follows that
  \begin{math}
    \| c_1^l \|
    \leq
    \| c_1^l - (\co[1]{\gamma})^l \| +
    | \co[1]{\gamma} |^l  \leq \delta + \epsilon.
  \end{math}
  At   this  stage,   we  shall   exploit  \eqref{eq:4}   once  again.
  Multiplying both sides of $(1 - c_1) P_n  = e_n$ by $1 + c_1 + \dots
  + c_1^{l-1}$, we get
  \begin{displaymath}
    ( 1 - c_1^l) P_n = (1 + c_1 + \dots + c_1^{l-1})e_n,
  \end{displaymath}
  or equivalently,
  \begin{displaymath}
    P_n = c_1^l P_n + (1 + c_1 + \dots + c_1^{l-1})e_n.
  \end{displaymath}
  Hence
  \begin{displaymath}
    \| P_n \| \leq (\delta + \epsilon) \| P_n \|
    + l \max\{1,  \| c_1 \|^{l-1}\} \| e_n \|,
  \end{displaymath}
  and, as $\lim_{n \to \infty} e_n = 0$,
  \begin{displaymath}
    \limsup_{n \to \infty} \| P_n \| \leq (\delta + \epsilon)
    \limsup_{n \to \infty} \| P_n \|.
  \end{displaymath}
  Remembering that $\delta + \epsilon < 1$, we conclude that
  \begin{math}
    \limsup_{n \to \infty} \| P_n \| = 0,
  \end{math}
  whence
  \begin{math}
    \lim_{n \to \infty} P_n = 0.
  \end{math}
\end{proof}

We now proceed to the proper proof of Theorem~\ref{thm:2}.

\begin{proof}[Proof of Theorem~\ref{thm:2}]
  By the cosine functional equation,
  \begin{displaymath}
    c_k^2 = \frac{1}{2}(1 + c_{2k})
    \quad
    \text{and}
    \quad
    (\co[k]{\gamma})^2 = \frac{1}{2}\big(1 + \co[2k]{\gamma}\big)
  \end{displaymath}
  for every $k \in \Z$.  Hence, for each $n \in \N$,
  \begin{multline}
    \label{eq:5}
      (c_1 - \co[1]{\gamma}) \sum_{k=0}^{n-1} (c_k - \co[k]{\gamma})^2
      \\ 
      \begin{aligned}
      & =  (c_1  -  \co[1]{\gamma})  \sum_{k=0}^{n-1}  \left(c_k^2  +
        (\co[k]{\gamma})^2 - 2 \co[k]{\gamma} c_k\right)
      \\
      & =   (c_1   -   \co[1]{\gamma})  \left[   n   +   \frac{1}{2}
        \sum_{k=0}^{n-1}   c_{2k}   +   \frac{1}{2}   \sum_{k=0}^{n-1}
        \co[2k]{\gamma}   -   2   \sum_{k=0}^{n-1}\co[k]{\gamma}   c_k
      \right].
      \end{aligned}
  \end{multline}
  As  a  first step  in  exploiting  the  above relation,  we  replace
  $\gamma$ by  $\gamma^{-1}$ in  \eqref{eq:3}, whereupon,  taking into
  account that $\co[1]{\gamma} = \co[1]{\gamma^{-1}}$, we find that
  \begin{displaymath}
    2 (\co[1]{\gamma} - c_1)\sum_{k=0}^{n-1} \gamma^{-k} c_k
    =
    \gamma^{-n} c_{n-1} - \gamma^{1-n} c_n - c_1 + \gamma
  \end{displaymath}
  for  each $n  \in \N$.   Adding  this identity  to \eqref{eq:3}  and
  dividing by $2$ yields
  \begin{displaymath}
    2 (\co[1]{\gamma} - c_1)\sum_{k=0}^{n-1} \co[k]{\gamma} c_k
    =
    \co[n]{\gamma} c_{n-1} - \co[n-1]{\gamma} c_n - c_1 + \co[1]{\gamma}.
  \end{displaymath}
  Hence, as  both $\{c_n\}_{n \in \Z}$  and $\{\co[n]{\gamma}\}_{n \in
    \Z}$ are bounded,
  \begin{equation}
    \label{eq:6}
    \lim_{n \to \infty}
    \frac{c_1 - \co[1]{\gamma}}{n} \sum_{k=0}^{n-1}\co[k]{\gamma} c_k = 0.
  \end{equation}
  
  Next, we apply Lemma~\ref{lem:1} to  the cosine sequences
  $\{c_{2n}\}_{n  \in  \Z}$  and  $\{\co[2n]{\gamma}\}_{n  \in  \Z}  =
  \{\co[n]{\gamma^2}\}_{n \in \Z}$, obtaining
  \begin{displaymath}
    \lim_{n \to \infty} \frac{1}{n}\sum_{k=0}^{n-1} c_{2k}
    =
    \begin{cases}
      1 & \text{if $\gamma^2 = 1$},
      \\
      0 & \text{otherwise}.
    \end{cases}
  \end{displaymath}
  By     applying    Lemma~\ref{lem:1}     to     two    copies     of
  $\{\co[2n]{\gamma}\}_{n \in \Z}$, or,  alternatively, by taking into
  account that
  \begin{displaymath}
    \sum_{k=0}^{n-1} \co[2k]{\gamma}
    =
    \frac{1}{2}
    \left[
      \sum_{k=0}^{n-1} \gamma^{2k} + \sum_{k=0}^{n-1} \gamma^{-2k}
    \right]
    =
    \begin{cases}
      1 & \text{if $\gamma^2 = 1$},
      \\
      \frac{1 - \gamma^{2n}}{2(1 - \gamma^2)} + \frac{1 -
        \gamma^{-2n}}{2(1 - \gamma^{-2})} & \text{otherwise},
    \end{cases}
  \end{displaymath}
  we also get
  \begin{displaymath}
    \lim_{n \to \infty} \frac{1}{n}\sum_{k=0}^{n-1} \co[2k]{\gamma}
    =
    \begin{cases}
      1 & \text{if $\gamma^2 = 1$},
      \\
      0 & \text{otherwise}.
    \end{cases}
  \end{displaymath}
  Hence
  \begin{displaymath}
    \lim_{n \to \infty}
    \frac{c_1 - \co[1]{\gamma}}{n} \left[
      n + \frac{1}{2} \sum_{k=0}^{n-1} c_{2k}
      + \frac{1}{2} \sum_{k=0}^{n-1} \co[2k]{\gamma} \right]
    =
    \begin{cases}
      2(c_1 - \co[1]{\gamma}) & \text{if $\gamma^2 = 1$},
      \\
      c_1 - \co[1]{\gamma} & \text{otherwise}.
    \end{cases}
  \end{displaymath}

  If  we  now  combine  the   above  relation  with  \eqref{eq:5}  and
  \eqref{eq:6}, we find that
  \begin{displaymath}
    \lim_{n \to \infty}
    \frac{c_1 - \co[1]{\gamma}}{n}
    \sum_{k=0}^{n-1} (c_k - \co[k]{\gamma})^2
    =
    \begin{cases}
      2(c_1 - \co[1]{\gamma}) & \text{if $\gamma^2 = 1$},
      \\
      c_1 - \co[1]{\gamma} & \text{otherwise}.
    \end{cases}
  \end{displaymath}
  Hence
  \begin{displaymath}
    \lim_{n \to \infty}
    \left\|
      \frac{c_1 - \co[1]{\gamma}}{n}
      \sum_{k=0}^{n-1} (c_k - \co[k]{\gamma})^2
    \right\|
    \geq \| c_1 - \co[1]{\gamma} \|.
  \end{displaymath}
  On the other hand, if $0 < \delta < 1$ is such that
  \begin{math}
    \| c_k - \co[k]{\gamma} \| \leq \delta
  \end{math}
  for every $k \in \Z$, then 
  \begin{displaymath}
    \left\|
      \frac{c_1 - \co[1]{\gamma}}{n}
      \sum_{k=0}^{n-1} (c_k - \co[k]{\gamma})^2
    \right\|
    \leq 
    \frac{\| c_1 - \co[1]{\gamma} \|}{n}
    \sum_{k=0}^{n-1} \| c_k - \co[k]{\gamma} \|^2
    \leq
    \delta^2 \|c_1 - \co[1]{\gamma} \|
  \end{displaymath}
  for each $n \in \N$.  Therefore,
  \begin{displaymath}
    \|c_1 - \co[1]{\gamma} \| \leq  \delta^2 \|c_1 - \co[1]{\gamma} \|,
  \end{displaymath}
  which implies  $\|c_1 - \co[1]{\gamma}  \| =  0$ and further  $c_1 =
  \co[1]{\gamma}$.  Hence, finally, $c_n  = \co[n]{\gamma}$ for all $n
  \in \Z$.
\end{proof}

\section{Proof of Theorem~\ref{thm:1}}
\label{sec:proof-theor-refthm:2}

Here we finally deduce Theorem~\ref{thm:1} from Theorem~\ref{thm:2}.

\begin{proof}[Proof of Theorem~\ref{thm:1}]
  Fix $g \in  G$ arbitrarily and define two  sequences $\{c_n\}_{n \in
    \Z}$ and $\{\tilde c_n\}_{n \in \Z}$ by
  \begin{displaymath}
    c_n = \Cos(ng)
    \quad
    \text{and}
    \quad
    \tilde c_n = c(ng)
  \end{displaymath}
  for every $n \in \Z$.   By a result of Kannappan~\cite{kannappan68},
  there exists $\gamma \in \CC \setminus \{0\}$ such that
  \begin{math}
    \tilde  c_n = \co[n]{\gamma}
  \end{math}
  for all $n \in \N$.  Now,  $\gamma$ has unit modulus, for otherwise,
  should  $|\gamma|  \neq  1$  hold, $\gamma^n  +  \gamma^{-n}$  would
  diverge in modulus to infinity  as $n \to \infty$, contradicting the
  boundedness of $\{c(g)\}_{g \in G}$.  Clearly,
  \begin{displaymath}
    \sup_{n \in \Z} \| c_n - \co[n]{\gamma} \| < 1,
  \end{displaymath}
  so  we  can  apply  Theorem~\ref{thm:2}  to  conclude  that  $c_n  =
  \co[n]{\gamma}$ for all $n \in \Z$.  In particular, $\Cos(g) = c_1 =
  \co[1]{\gamma} = c(g)$.  As $g$  was chosen arbitrarily, the theorem
  is established.
\end{proof}

\bibliographystyle{amsplainurl}
\bibliography{coswallen-bib}

\end{document}